\documentclass[final,leqno]{siamltex704}
\usepackage{amsmath}
\usepackage{amssymb}
\usepackage{graphicx}
\usepackage[notcite,notref]{showkeys}
\usepackage{tikz}
\newcounter{dummy} \numberwithin{dummy}{section}
\newtheorem{defi}[dummy]{Definition}

\newtheorem{algorithm}{Weak Galerkin Algorithm}

\renewcommand{\ldots}{\dotsc}
\newcommand{\bu}{{\bf u}}

\newcommand{\bw}{{\bf w}}

\newcommand{\be}{{\bf e}}
\newcommand{\bv}{{\bf v}}

\def\T{{\mathcal T}}
\def\E{{\mathcal E}}
\def\V{{\mathcal V}}

\def\pT{{\partial T}}
\def\l{{\langle}}
\def\r{{\rangle}}

\def\dw{{\mathcal \nabla_w\cdot}}

\def\T{{\mathcal T}}
\def\E{{\mathcal E}}

\def\bbf{{\bf f}}
\def\bg{{\bf g}}
\def\bn{{\bf n}}

\def\bbQ{\mathbb{Q}}

\def\3bar{{|\hspace{-.02in}|\hspace{-.02in}|}}

\title{A Weak Galerkin Finite Element Method for the Stokes Equations}
\author{Junping Wang\thanks{Division of Mathematical Sciences, National
Science Foundation, Arlington, VA 22230 (jwang@\break nsf.gov). The
research of Wang was supported by the NSF IR/D program, while
working at the Foundation. However, any opinion, finding, and
conclusions or recommendations expressed in this material are those
of the author and do not necessarily reflect the views of the
National Science Foundation.} \and Xiu Ye\thanks{Department of
Mathematics, University of Arkansas at Little Rock, Little Rock, AR
72204 (xxye@ualr.edu). This research was supported in part by
National Science Foundation Grant DMS-1115097}}
\begin{document}
\maketitle

\begin{abstract}
This paper introduces a weak Galerkin (WG) finite element method for
the Stokes equations in the primary velocity-pressure formulation.
This WG method is equipped with stable finite elements consisting of
usual polynomials of degree $k\ge 1$ for the velocity and
polynomials of degree $k-1$ for the pressure, both are
discontinuous. The velocity element is enhanced by polynomials of
degree $k-1$ on the interface of the finite element partition. All
the finite element functions are discontinuous for which the usual
gradient and divergence operators are implemented as distributions
in properly-defined spaces. Optimal-order error estimates are
established for the corresponding numerical approximation in various
norms. It must be emphasized that the WG finite element method is
designed on finite element partitions consisting of arbitrary shape
of polygons or polyhedra which are shape regular.
\end{abstract}


\pagestyle{myheadings}

\section{Introduction}
In this paper, we are concerned with the development of weak
Galerkin (WG) finite element methods for the Stokes problem which
seeks unknown functions $\bu$ and $p$ satisfying
\begin{eqnarray}
-\Delta\bu+\nabla p&=& \bbf\quad
\mbox{in}\;\Omega,\label{moment}\\
\nabla\cdot\bu &=&0\quad \mbox{in}\;\Omega,\label{cont}\\
\bu&=&\bg\quad \mbox{on}\; \partial\Omega,\label{bc}
\end{eqnarray}
where $\Omega$ is a polygonal or polyhedral domain in
$\mathbb{R}^d\; (d=2,3)$.

The weak form in the primary velocity-pressure formulation for the
Stokes problem (\ref{moment})--(\ref{bc}) seeks $\bu\in
[H^1(\Omega)]^d$ and $p\in L_0^2(\Omega)$ satisfying $\bu=\bg$ on
$\partial\Omega$ and
\begin{eqnarray}
(\nabla\bu,\nabla\bv)-(\nabla\cdot\bv, p)&=&({\bf f}, \bv),\label{w1}\\
(\nabla\cdot\bu, q)&=&0,\label{w2}
\end{eqnarray}
for all $\bv\in [H_0^1(\Omega)]^d$ and $q\in L_0^2(\Omega)$. The
conforming finite element method for (\ref{moment})--(\ref{bc})
developed over the last several decades is based on the weak
formulation (\ref{w1})-(\ref{w2}) by constructing a pair of finite
element spaces satisfying the {\em inf-sup} condition of Babu\u{s}ka
\cite{babuska} and Brezzi \cite{brezzi}. Readers are referred to
\cite{gr} for specific examples and details in the classical finite
element methods for the Stokes equations.

Weak Galerkin refers to a general finite element technique for
partial differential equations in which differential operators are
approximated by weak forms as distributions for generalized
functions. Thus, two of the key features in weak Galerkin methods
are (1) the approximating functions are discontinuous, and (2) the
usual derivatives are taken as distributions or approximations of
distributions. A weak Galerkin method was first introduced and
analyzed for second order elliptic equations in \cite{wy} and in a
conference in Nankai University in the summer of 2011 by one of the
authors. The objective of this paper is to develop a weak Galerkin
finite element method for (\ref{moment})-(\ref{bc}) that is
efficient and robust by allowing the use of discontinuous
approximating functions on finite element partitions consisting of
arbitrary polygons or polyhedra with certain shape regularity.

In general, weak Galerkin finite element formulations for partial
differential equations can be derived naturally by replacing usual
derivatives by weakly-defined derivatives in the corresponding
variational forms, with the option of adding a stabilization term to
enforce a weak continuity of the approximating functions. For the
Stokes problem (\ref{moment})-(\ref{bc}) interpreted by the
variational formulation (\ref{w1})-(\ref{w2}), the two principle
differential operators are the gradient and the divergence operator
defined in the Sobolev space $[H^1(\Omega)]^d$. Formally, our weak
Galerkin method for the Stokes problem would take the following
form: Find $\bu_h$ and $p_h$ from properly-defined finite element
spaces satisfying
\begin{eqnarray}
(\nabla_w\bu_h,\nabla_w\bv)-(\nabla_w\cdot\bv,p_h)+s(\bu_h,\bv)&=&({\bf f},\bv),\label{w3}\\
(\nabla_w\cdot\bu_h,q)&=&0\label{w4}
\end{eqnarray}
for all test functions $\bv$ and $q$ in test spaces. Here $\nabla_w$
is a weak gradient and $\nabla_w\cdot$ is a weak divergence operator
to be detailed in this study (see Section \ref{Section:wg-wd} for
details). The bilinear form $s(\cdot,\cdot)$ in (\ref{w3}) is a
parameter free stabilizer that shall enforce a certain weak
continuity for the underlying approximating functions. The use of
totally discontinuous functions and weak derivatives in the WG
formulation provides the numerical scheme with many nice features.
First, the construction of stable elements for the Stokes equations
under WG formulation is straight forward with standard polynomials.
Secondly, the WG method allows the use of finite element partitions
with arbitrary shape of polygons in 2D or polyhedra in 3D with
certain shape regularity. The later property provides a convenient
and useful flexibility in both numerical approximation and mesh
generation. Thirdly, our WG formulation is parameter-free and has
competitive number of unknowns since lower degree of polynomials are
used on element boundaries, and the unknowns corresponding to the
interior of each element can be eliminated from the system.

The paper is organized as follows. In Section
\ref{Section:preliminaries}, we introduce some standard notations in
Sobolev spaces. Two weakly-defined differential operators, weak
gradient and weak divergence, are introduced in Section
\ref{Section:wg-wd}. The WG finite element scheme for the Stokes
problem (\ref{moment})-(\ref{cont}) is developed in Section
\ref{Section:wg-fem}. In Section \ref{Section:stability}, we shall
study the stability and solvability of the WG scheme. In particular,
the usual {\em inf-sup} condition is established for the WG scheme.
In Section \ref{Section:error-equations}, we shall derive an error
equation for the WG approximations. Optimal-order error estimates
for the WG finite element approximations are derived in Section
\ref{Section:error-analysis} in virtually an $H^1$ norm for the
velocity, and $L^2$ norm for both the velocity and the pressure. In
Section \ref{Section:concluding}, we make some concluding remarks by
mentioning some outstanding issues for future consideration.
Finally, we present some technical estimates for quantities related
to the local $L^2$ projections into various finite element spaces in
Appendix \ref{appendix-1}.

\section{Preliminaries and Notations}\label{Section:preliminaries}

Let $D$ be any open bounded domain with Lipschitz continuous
boundary in $\mathbb{R}^d, d=2, 3$. We use the standard definition
for the Sobolev space $H^s(D)$ and the associated inner product
$(\cdot,\cdot)_{s,D}$, norm $\|\cdot\|_{s,D}$, and seminorm
$|\cdot|_{s,D}$ for any $s\ge 0$. For example, for any integer $s\ge
0$, the seminorm $|\cdot|_{s, D}$ is given by
$$
|v|_{s, D} = \left( \sum_{|\alpha|=s} \int_D |\partial^\alpha v|^2
dD \right)^{\frac12}
$$
with the usual notation
$$
\alpha=(\alpha_1, \ldots, \alpha_d), \quad |\alpha| =
\alpha_1+\ldots+\alpha_d,\quad
\partial^\alpha =\prod_{j=1}^d\partial_{x_j}^{\alpha_j}.
$$
The Sobolev norm $\|\cdot\|_{m,D}$ is given by
$$
\|v\|_{m, D} = \left(\sum_{j=0}^m |v|^2_{j,D} \right)^{\frac12}.
$$

The space $H^0(D)$ coincides with $L^2(D)$, for which the norm and
the inner product are denoted by $\|\cdot \|_{D}$ and
$(\cdot,\cdot)_{D}$, respectively. When $D=\Omega$, we shall drop
the subscript $D$ in the norm and inner product notation.

The space $H({\rm div};D)$ is defined as the set of vector-valued
functions on $D$ which, together with their divergence, are square
integrable; i.e.,
\[
H({\rm div}; D)=\left\{ \bv: \ \bv\in [L^2(D)]^d, \nabla\cdot\bv \in
L^2(D)\right\}.
\]
The norm in $H({\rm div}; D)$ is defined by
$$
\|\bv\|_{H({\rm div}; D)} = \left( \|\bv\|_{D}^2 + \|\nabla
\cdot\bv\|_{D}^2\right)^{\frac12}.
$$

\section{Weak Differential Operators and Their Approximations}\label{Section:wg-wd}

The key to weak Galerkin methods is the use of weak derivatives in
the place of strong derivatives that define the weak formulation for
the underlying partial differential equations. The two differential
operators used in the weak formulation (\ref{w1}) and (\ref{w2}) are
gradient and divergence. Thus, it is essential to introduce a weak
version for both the gradient and the divergence operator. In
\cite{wy-mixed}, a weak divergence operator has been introduced and
employed to the mixed formulation of second order elliptic
equations. In \cite{wy} and \cite{mwy}, a weak gradient operator was
introduced for scalar functions. Those weakly defined differential
operators shall be employed to the Stokes problem
(\ref{w1})-(\ref{w2}) in a weak Galerkin approximation. For
convenience, the rest of the section will review the definition for
the weak gradient and the weak divergence, respectively. Note that
the weak gradient shall be applied to each component when the
underlying function is vector-valued, as is the case for the Stokes
problem.

\subsection{Weak gradient and discrete weak gradient}

Let $K$ be any  polygonal or polyhedral domain with boundary
$\partial K$. A  weak vector-valued function on the region $K$
refers to a vector-valued function $\bv=\{\bv_0, \bv_b\}$ such that
$\bv_0\in [L^2(K)]^d$ and $\bv_b\in [H^{\frac12}(\partial K)]^d$.
The first component $\bv_0$ can be understood as the value of $\bv$
in $K$, and the second component $\bv_b$ represents $\bv$ on the
boundary of $K$. Note that $\bv_b$ may not necessarily be related to
the trace of $\bv_0$ on $\partial K$ should a trace be well-defined.
Denote by $\V(K)$ the space of weak functions on $K$; i.e.,
\begin{equation}\label{hi.888}
\V(K)= \{\bv=\{\bv_0, \bv_b \}:\ \bv_0\in [L^2(K)]^d,\; \bv_b\in
[H^{\frac12}(\partial K)]^d\}.
\end{equation}
The weak gradient operator is defined as follows.
\medskip

\begin{defi}
The dual of $[L^2(K)]^d$ can be identified with itself by using the
standard $L^2$ inner product as the action of linear functionals.
With a similar interpretation, for any $\bv\in \V(K)$, the weak
gradient of $\bv$ is defined as a linear functional $\nabla_w \bv$ in
the dual space of $[H(div,K)]^d$ whose action on each $q\in [H(div,K)]^d$ is
given by
\begin{equation}\label{wg}
(\nabla_w \bv, q)_K = -(\bv_0, \nabla\cdot q)_K + \langle \bv_b,
q\cdot\bn\rangle_{\partial K},
\end{equation}
where $\bn$ is the outward normal direction to $\partial K$,
$(\bv_0,\nabla\cdot q)_K=\int_K \bv_0 (\nabla\cdot q)dK$ is the action
of $\bv_0$ on $\nabla\cdot q$, and $\langle \bv_b,
q\cdot\bn\rangle_{\partial K}=\int_{\partial K}\bv_bq\cdot\bn ds$ is the action of $q\cdot\bn$ on
$\bv_b\in [H^{\frac12}(\partial K)]^d$.
\end{defi}

\medskip

The Sobolev space $[H^1(K)]^d$ can be embedded into the space $\V(K)$ by
an inclusion map $i_\V: \ [H^1(K)]^d\to \V(K)$ defined as follows
$$
i_\V(\phi) = \{\phi|_{K}, \phi|_{\partial K}\},\qquad \phi\in [H^1(K)]^d.
$$
With the help of the inclusion map $i_\V$, the Sobolev space $[H^1(K)]^d$
can be viewed as a subspace of $\V(K)$ by identifying each $\phi\in
[H^1(K)]^d$ with $i_\V(\phi)$.

Let $P_{r}(K)$ be the set of polynomials on $K$ with degree no more
than $r$.

\begin{defi}
The discrete weak gradient operator, denoted by
$\nabla_{w,r, K}$, is defined as the unique polynomial
$(\nabla_{w,r, K}\bv) \in [P_r(K)]^{d\times d}$ satisfying the
following equation,
\begin{equation}\label{d-g}
(\nabla_{w,r, K}\bv, q)_K = -(\bv_0,\nabla\cdot q)_K+ \langle \bv_b,
q\cdot\bn\rangle_{\partial K},\qquad \forall q\in [P_r(K)]^{d\times d}.
\end{equation}
\end{defi}

\subsection{Weak divergence and discrete weak divergence}

To define weak divergence, we require weak function $\bv=\{\bv_0,
\bv_b\}$ such that $\bv_0\in [L^2(K)]^d$ and $\bv_b\cdot\bn\in
H^{-\frac12}(\partial K)$. Denote by $V(K)$ the space of weak vector-valued functions on $K$;
i.e.,
\begin{equation}\label{hi.999}
{V}(K) = \{\bv=\{\bv_0, \bv_b \}:\ \bv_0\in [L^2(K)]^d,\;
\bv_b\cdot\bn\in H^{-\frac12}(\partial K)\}.
\end{equation}
A weak divergence operator can be defined as follows.
\medskip

\begin{defi}
The dual of $L^2(K)$ can be identified with itself by using the
standard $L^2$ inner product as the action of linear functionals.
With a similar interpretation, for any $\bv\in V(K)$, the weak
divergence of $\bv$ is defined as a linear functional $\nabla_w
\cdot\bv$ in the dual space of $H^1(K)$ whose action on each
$\varphi\in H^1(K)$ is given by
\begin{equation}\label{wd}
(\nabla_w\cdot\bv, \varphi)_K = -(\bv_0, \nabla\varphi)_K + \langle
\bv_b\cdot\bn, \varphi\rangle_{\partial K},
\end{equation}
where $\bn$ is the outward normal direction to $\partial K$,
$(\bv_0,\nabla\varphi)_K$ is the
action of $\bv_0$ on $\nabla\varphi$, and $\langle \bv_b\cdot\bn,
\varphi\rangle_{\partial K}$ is the action of $\bv_b\cdot\bn$ on
$\varphi\in H^{\frac12}(\partial K)$.
\end{defi}

The Sobolev space $[H^1(K)]^d$ can be embedded into the space $V(K)$
by an inclusion map $i_V: \ [H^1(K)]^d\to V(K)$ defined as follows
$$
i_V(\phi) =\{\phi|_{K},\ \phi|_{\partial K}\},\qquad \phi\in
[H^1(K)]^d.
$$
\begin{defi}
A discrete weak divergence operator, denoted by
$\nabla_{w,r,K}\cdot$, is defined as the unique polynomial
$(\nabla_{w,r,K}\cdot\bv) \in P_r(K)$ that satisfies the following
equation
\begin{equation}\label{d-d}
(\nabla_{w,r,K}\cdot\bv, \varphi)_K = -(\bv_0, \nabla\varphi)_K + \langle
\bv_b\cdot\bn, \varphi\rangle_{\partial K},\qquad
\forall \varphi\in P_r(K).
\end{equation}
\end{defi}

\section{A Weak Galerkin Finite Element Scheme}\label{Section:wg-fem}

Let ${\cal T}_h$ be a  partition of the domain $\Omega$ with mesh
size $h$ that consists of arbitrary polygons/polyhedra. Assume that the
partition ${\cal T}_h$ is WG shape regular - defined by a set of
conditions as detailed in \cite{wy-mixed} and \cite{mwy}. Denote by
${\cal E}_h$ the set of all edges/flat faces in ${\cal T}_h$, and let
${\cal E}_h^0={\cal E}_h\backslash\partial\Omega$ be the set of all
interior edges/faces.

For any integer $k\ge 1$, we define a weak Galerkin finite element space for
the velocity variable as follows
\[
V_h=\left\{ \bv=\{\bv_0, \bv_b\}:\ \{\bv_0, \bv_b\}|_{T}\in
[P_{k}(T)]^d\times [P_{k-1}(e)]^d,\ e\subset\pT\right\}.
\]
We would like to emphasize that there is only a single value $\bv_b$
defined on each edge $e\in\E_h$. For the pressure variable, we have
the following finite element space
\[
W_h=\left\{q:\ q\in L_0^2(\Omega), \ q|_T\in P_{k-1}(T)\right\}.
\]
Denote by $V_h^0$ the subspace of $V_h$ consisting of discrete weak
functions with vanishing boundary value; i.e.,
\[
V_h^0=\left\{\bv=\{\bv_0, \bv_b\}\in V_h, \bv_b=0\ \mbox{on}\
\partial\Omega \right\}.
\]
The discrete weak gradient $\nabla_{w,k-1}$ and the discrete weak
divergence $(\nabla_{w,k-1}\cdot)$ on the finite element space $V_h$
can be computed by using (\ref{d-g}) and (\ref{d-d}) on each element
$T$, respectively. More precisely, they are given by
\begin{eqnarray*}
(\nabla_{w,k-1}\bv)|_T &=&\nabla_{w,k-1, T} (\bv|_T),\qquad \forall\
\bv \in V_h,\\
(\nabla_{w,k-1}\cdot\bv)|_T &=&\nabla_{w,k-1, T}\cdot
(\bv|_T),\qquad \forall \ \bv \in V_h.
\end{eqnarray*}
For simplicity of notation, from now on we shall drop the subscript
$k-1$ in the notation $\nabla_{w,k-1}$ and $(\nabla_{w,k-1}\cdot)$
for the discrete weak gradient and the discrete weak divergence. The
usual $L^2$ inner product can be written locally on each element as
follows
\begin{eqnarray*}
(\nabla_w\bv,\ \nabla_w\bw)&=&\sum_{T\in\T_h}(\nabla_w\bv,\ \nabla_w\bw)_T,\\
(\nabla_w\cdot\bv,\ q)&=&\sum_{T\in\T_h}(\nabla_w\cdot\bv,\ q)_T.
\end{eqnarray*}

Denote by $Q_{0}$ the $L^2$ projection operator from $[L^2(T)]^d$
onto $[P_k(T)]^d$. For each edge/face $e\in {\cal E}_h$, denote by
$Q_{b}$ the $L^2$ projection from $[L^2(e)]^d$ onto $[P_{k-1}(e)]^d$.
We shall combine $Q_0$ with $Q_b$ by writing $Q_h=\{Q_0,Q_b\}$.

We are now in a position to describe a weak Galerkin finite element
scheme for the Stokes equations (\ref{moment})-(\ref{bc}). To this
end, we first introduce three bilinear forms as follows
\begin{eqnarray*}
s(\bv,\;\bw) &= &\sum_{T\in {\cal T}_h}h_T^{-1}\l Q_b\bv_0-\bv_b,\;\;Q_b\bw_0-\bw_b\r_\pT,\\
a(\bv,\ \bw)&=&(\nabla_w\bv,\ \nabla_w\bw)+s(\bv,\bw),\\
b(\bv,\ q)&=&(\nabla_w\cdot\bv,\ q).
\end{eqnarray*}

\begin{algorithm}
A numerical approximation for (\ref{moment})-(\ref{bc})
can be obtained by seeking $\bu_h=\{\bu_0,\bu_b\}\in V_h$ and
$p_h\in W_h$ such that $\bu_b=Q_b\bg$ on $\partial\Omega$ and
\begin{eqnarray}
a(\bu_h,\ \bv)-b(\bv,\;p_h)&=&(f,\;\bv_0),\label{wg1}\\
b(\bu_h,\;q)&=&0,\label{wg2}
\end{eqnarray}
for all $\bv=\{\bv_0,\bv_b\}\in V_h^0$ and $q\in W_h$.
\end{algorithm}

\section{Stability and Solvability}\label{Section:stability}

The WG finite element scheme (\ref{wg1})-(\ref{wg2}) is a typical
saddle-point problem which can be analyzed by using the well known
theory developed by Babu\u{s}ka \cite{babuska} and Brezzi
\cite{brezzi}. The core of the theory is to verify two properties:
(1) boundedness and a certain coercivity for the bilinear form
$a(\cdot,\cdot)$, and (2) boundedness and {\em inf-sup} condition
for the bilinear form $b(\cdot,\cdot)$.

The finite element space $V_h^0$ is a normed linear space with a
triple-bar norm given by
\begin{equation}\label{3barnorm}
\3bar
\bv\3bar^2=\sum_{T\in\T_h}\|\nabla_w\bv\|_T^2+\sum_{T\in\T_h}h_T^{-1}\|Q_b\bv_0-\bv_b\|_{\partial
T}^2.
\end{equation}
We claim that $\3bar \cdot \3bar$ indeed provides a norm in $V_h^0$.
For simplicity, we shall only verify the positive length property
for $\3bar \cdot \3bar$. Assume that $\3bar \bv \3bar=0$ for some
$\bv\in V_h^0$. It follows that
$$
0=(\nabla_w\bv,\nabla_w\bv)+\sum_{T\in\T_h}h_T^{-1}\l
Q_b\bv_0-\bv_b,\ Q_b\bv_0-\bv_b\r_\pT,
$$
which implies that $\nabla_w \bv=0$ on each element $T$ and
$Q_b\bv_0=\bv_b$ on $\pT$. Thus, we have from the definition
(\ref{d-g}) that for any $\tau\in[P_{k-1}(T)]^{d\times d}$
\begin{eqnarray*}
0&=&(\nabla_w \bv,\tau)_T\\
&=&-(\bv_0,\nabla\cdot \tau)_T+\langle \bv_b,\tau\cdot\bn\rangle_\pT\\
&=&(\nabla \bv_0,\tau)_T-\langle \bv_0-\bv_b,
\tau\cdot\bn\rangle_\pT\\
&=&(\nabla \bv_0,\tau)_T-\langle Q_b\bv_0-\bv_b,
\tau\cdot\bn\rangle_\pT\\
&=&(\nabla \bv_0,\tau)_T.
\end{eqnarray*}
Letting $\tau=\nabla \bv_0$ in the equation above yields $\nabla
\bv_0=0$ on $T\in {\cal T}_h$. It follows that $\bv_0=const$ on
every $T\in\T_h$. This, together with the fact that $Q_b\bv_0=\bv_b$
on $\partial T$ and $\bv_b=0$ on $\partial\Omega$, implies that
$\bv_0=0$ and $\bv_b=0$.

\smallskip
Note that $\3bar \cdot \3bar$ defines only a semi-norm in $V_h$. It
is not hard to see that $a(\bv,\;\bv)=\3bar \bv\3bar^2$ for any
$\bv\in V_h$. In fact, the trip-bar norm is equivalent to the
standard $H^1$-norm, but was defined for weak finite element
functions. It follows from the definition of $\3bar\cdot\3bar$ and
the usual Cauchy-Schwarz inequality that the following boundedness
and coercivity hold true for the bilinear form $a(\cdot,\cdot)$.

\begin{lemma}\label{coercivity+boundedness}
For any $\bv,\bw\in V_h^0$, we have
\begin{eqnarray}
|a(\bv,\bw)|&\le& \3bar\bv\3bar\3bar\bw\3bar,\label{bd}\\
a(\bv,\bv)&=&\3bar\bv\3bar^2.\label{elliptic}
\end{eqnarray}
\end{lemma}

\medskip

In addition to the projection $Q_h=\{Q_0,Q_b\}$ defined in the
previous section, let $\bbQ_h$ and ${\bf Q}_h$ be two local $L^2$
projections onto $P_{k-1}(T)$ and $[P_{k-1}(T)]^{d\times d}$,
respectively.

\begin{lemma}\label{lem-1-0}
The projection operators $Q_h$, ${\bf Q}_h$, and $\bbQ_h$ satisfy
the following commutative properties
\begin{eqnarray}\label{key}
\nabla_w (Q_h \bv) &=& {\bf Q}_h (\nabla\bv),\qquad\forall \ \bv\in
[H^1(\Omega)]^d,\\
\label{key11} \nabla_w\cdot (Q_h\bv)
&=&\bbQ_h(\nabla\cdot\bv),\qquad\forall\ \bv\in H({\rm div},\Omega).
\end{eqnarray}
\end{lemma}
\begin{proof}
Using (\ref{d-g}), we have
$$
(\nabla_w (Q_h \bv),\; q)_T = -(Q_0\bv,\; \nabla\cdot q)_T +\langle
Q_b\bv,\; q\cdot\bn\rangle_{\pT}
$$
for all $q\in [P_{k-1}(T)]^{d\times d}$. Next, we use the definition
of $Q_h$ and ${\bf Q}_h$ and the usual integration by parts to
obtain
\begin{eqnarray*}
-(Q_0\bv,\; \nabla\cdot q)_T +\langle Q_b\bv,\;
q\cdot\bn\rangle_{\pT}
&=&-(\bv,\; \nabla\cdot q)_T + \langle \bv,\; q\cdot\bn\rangle_{\partial T}\\
&=&(\nabla\bv,\; q)\\
&=&({\bf Q}_h(\nabla \bv),\; q).
\end{eqnarray*}
Thus,
$$
(\nabla_w (Q_h \bv),\; q)_T =({\bf Q}_h(\nabla \bv),\; q),\qquad
\forall \ q\in [P_{k-1}(T)]^{d\times d},
$$
which verifies the identity (\ref{key}).

To verify (\ref{key11}), we use the discrete weak divergence
(\ref{d-d}) to obtain
$$
(\nabla_w\cdot (Q_h \bv),\; \varphi)_T = -(Q_0\bv,\; \nabla
\varphi)_T +\langle Q_b\bv \cdot\bn,\; \varphi\rangle_{\pT}
$$
for all $\varphi\in P_{k-1}(T)$. Next, we use the definition of
$Q_h$ and $\bbQ_h$ and the usual integration by parts to arrive at
\begin{eqnarray*}
-(Q_0\bv,\; \nabla \varphi)_T +\langle Q_b\bv\cdot\bn,\;
\varphi\rangle_{\pT}
&=&-(\bv,\; \nabla \varphi)_T + \langle \bv\cdot\bn,\; \varphi\rangle_{\partial T}\\
&=&(\nabla\cdot\bv,\; \varphi)_T\\
&=&({\bbQ}_h(\nabla\cdot\bv),\; \varphi)_T.
\end{eqnarray*}
It follows that
$$
(\nabla_w\cdot (Q_h \bv),\; \varphi)_T =({\bbQ}_h(\nabla\cdot\bv),\;
\varphi)_T,\qquad\forall\ \varphi\in P_{k-1}(T).
$$
This completes the proof of (\ref{key11}), and hence the lemma.
\end{proof}

\medskip

For the bilinear form $b(\cdot,\cdot)$, we have the following result
on the {\em inf-sup} condition.

\smallskip
\begin{lemma}\label{Lemma:inf-sup}
There exists a positive constant $\beta$ independent of $h$ such
that
\begin{equation}\label{inf-sup}
\sup_{\bv\in V_h^0}\frac{b(\bv,\rho)}{\3bar\bv\3bar}\ge \beta
\|\rho\|
\end{equation}
for all $\rho\in W_h$.
\end{lemma}

\begin{proof}
For any given $\rho\in W_h\subset L_0^2(\Omega)$, it is well known
\cite{bs,bf,cr,gr,gun} that there exists a vector-valued function
$\tilde\bv\in [H_0^1(\Omega)]^d$ such that
\begin{equation}\label{c-inf-sup}
\frac{(\nabla\cdot\tilde\bv,\rho)}{\|\tilde\bv\|_1}\ge C\|\rho\|,
\end{equation}
where $C>0$ is a constant depending only on the domain $\Omega$. By
setting $\bv=Q_h\tilde{\bv}\in V_h$, we claim that the following
holds true
\begin{equation}\label{m9}
\3bar\bv\3bar\le C_0 \|\tilde{\bv}\|_1
\end{equation}
for some constant $C_0$. To this end, we use equation (\ref{key}) to
obtain
\begin{equation}\label{feb-8.01}
\sum_{T\in\T_h}\|\nabla_w\bv\|^2_T=\sum_{T\in\T_h}\|\nabla_w(Q_h\tilde{\bv})\|^2_T
=\sum_{T\in\T_h}\|{\bf Q}_h\nabla \tilde{\bv}\|^2_T\le
\|\nabla\tilde\bv\|^2.
\end{equation}
Next, we use (\ref{trace}), (\ref{Qh}), and the definition of $Q_b$
to obtain
\begin{eqnarray}
\nonumber\sum_{T\in\T_h}h^{-1}_T\|Q_b\bv_0-\bv_b\|_\pT^2&=&\sum_{T\in\T_h}h^{-1}_T\|Q_b(Q_0\tilde{\bv})-Q_b\tilde{\bv}\|_\pT^2\\
&=&\sum_{T\in\T_h}h^{-1}_T\|Q_b(Q_0\tilde{\bv}-\tilde{\bv})\|_\pT^2\nonumber\\
&\le& \sum_{T\in\T_h}h^{-1}_T\|Q_0\tilde{\bv}-\tilde{\bv}\|_\pT^2\nonumber\\
&\le& C\sum_{T\in\T_h}(h^{-2}_T\|Q_0\tilde{\bv}-\tilde{\bv}\|_T^2+
\|\nabla(Q_0\tilde{\bv}-\tilde{\bv})\|_T^2)\nonumber\\
&\le& C\|\nabla \tilde{\bv}\|^2.\label{feb-8.02}
\end{eqnarray}
Combining the estimate (\ref{feb-8.01}) with (\ref{feb-8.02}) yields
the desired inequality (\ref{m9}).

It follows from (\ref{key11}) and the definition of $\bbQ_h$ that
\begin{eqnarray*}
b(\bv,\;\rho)&=&(\dw
(Q_h\tilde\bv),\;\rho)=(\bbQ_h(\nabla\cdot\tilde\bv),\;\rho)=(\nabla\cdot\tilde\bv,\;\rho).
\end{eqnarray*}
Using the above equation, (\ref{c-inf-sup}) and (\ref{m9}), we have
\begin{eqnarray*}
\frac{|b(\bv,\rho)|} {\3bar\bv\3bar} &\ge &
\frac{|(\nabla\cdot\tilde\bv,\rho)|}{C_0\|\tilde\bv\|_1}\ge
\beta\|\rho\|
\end{eqnarray*}
for a positive constant $\beta$. This completes the proof of the
lemma.
\end{proof}

\smallskip

It follows from Lemma \ref{coercivity+boundedness} and Lemma
\ref{Lemma:inf-sup} that the following solvability holds true for
the weak Galerkin finite element scheme (\ref{wg1})-(\ref{wg2}).

\smallskip
\begin{lemma}
The weak Galerkin finite element scheme (\ref{wg1})-(\ref{wg2}) has
one and only one solution.
\end{lemma}

\section{Error Equations}\label{Section:error-equations}
Let $\bu_h=\{\bu_0,\bu_b\}\in V_h$ and $p_h\in W_h$ be the weak
Galerkin finite element solution arising from the numerical scheme
(\ref{wg1})-(\ref{wg2}). Denote by $\bu$ and $p$ the exact solution
of (\ref{moment})-(\ref{bc}). The $L^2$ projection of $\bu$ in the
finite element space $V_h$ is given by
$$
Q_h\bu=\{Q_0\bu,Q_b\bu\}.
$$
Similarly, the pressure $p$ is projected into $W_h$ as $\bbQ_h p$.
Denote by $\be_h$ and $\varepsilon_h$ the corresponding error given
by
\begin{equation}\label{EQ:errors}
\be_h=\{\be_0,\;\be_b\}=\{Q_0\bu-\bu_0,\;Q_b\bu-\bu_b\},\qquad
\varepsilon_h=\bbQ_hp-p_h.
\end{equation}
The goal of this section is to derive two equations for which the
error $\be_h$ and $\varepsilon_h$ shall satisfy. The resulting
equations are called {\em error equations}, which play a critical
role in the convergence analysis for the weak Galerkin finite
element method.

\begin{lemma}\label{Lemma:QhuEquation}
Let $(\bw;\rho)\in [H^1(\Omega)]^d\times L^2(\Omega)$ be
sufficiently smooth and satisfy the following equation
\begin{equation}\label{new-moment}
-\Delta \bw + \nabla \rho = \eta
\end{equation}
in the domain $\Omega$. Let $Q_h\bw=\{Q_0\bw,Q_b\bw\}$ and $\bbQ_h
\rho$ be the $L^2$ projection of $(\bw; \rho)$ into the finite
element space $V_h\times W_h$. Then, the following equation holds
true
\begin{equation}\label{sf-01}
(\nabla_w(Q_h\bw),\nabla_w\bv)-(\dw\bv,\bbQ_h \rho)=(\eta,
\bv_0)+\ell_{\bw}(\bv)-\theta_\rho(\bv)
\end{equation}
for all $\bv\in V_h^0$, where $\ell_{\bw}$ and $\theta_\rho$ are two
linear functionals on $V_h^0$ defined by
\begin{eqnarray*}
\ell_{\bw}(\bv)&=&\sum_{T\in\T_h}\l\bv_0-\bv_b,\ \nabla\bw\cdot\bn-{\bf Q}_h(\nabla\bw)
\cdot\bn\r_\pT,\\
\theta_\rho(\bv)&=&\sum_{T\in\T_h}\langle \bv_0-\bv_b,\
(\rho-\bbQ_h\rho)\bn\rangle_\pT.
\end{eqnarray*}
\end{lemma}

\begin{proof} First, it follows from (\ref{key}), (\ref{d-g}), and the integration by parts that
\begin{eqnarray}
(\nabla_w(Q_h\bw),\nabla_w\bv)_T&=&({\bf Q}_h(\nabla\bw),\nabla_w\bv)_T\nonumber\\
&=&-(\bv_0,\nabla\cdot {\bf Q}_h(\nabla \bw))_T+\langle \bv_b,{\bf Q}_h(\nabla\bw)
\cdot\bn\rangle_\pT\nonumber\\
&=&(\nabla\bv_0,{\bf Q}_h(\nabla \bw))_T-\langle \bv_0-\bv_b,\ {\bf Q}_h(\nabla\bw)
\cdot\bn\rangle_\pT\nonumber\\
&=&(\nabla\bw,\nabla\bv_0)_T-\l \bv_0-\bv_b,{\bf
Q}_h(\nabla\bw)\cdot\bn\r_\pT.\label{m1}
\end{eqnarray}
Next, by using (\ref{key11}), (\ref{d-d}), the fact that
$\sum_{T\in\T_h}\langle \bv_b,p\;\bn\rangle_\pT=0$ and the
integration by parts, we obtain
\begin{eqnarray*}
(\dw\bv,\bbQ_h\rho)&=&-\sum_{T\in\T_h}(\bv_0, \nabla (\bbQ_h\rho))_T+
\sum_{T\in\T_h}\l\bv_b,(\bbQ_h\rho)\bn\r_\pT\nonumber\\
&=&\sum_{T\in\T_h}(\nabla\cdot\bv_0,\bbQ_h\rho)_T-\sum_{T\in\T_h}\langle \bv_0-\bv_b,(\bbQ_h\rho)
\bn\rangle_\pT\nonumber\\
&=&\sum_{T\in\T_h}(\nabla\cdot\bv_0,\rho)_T-\sum_{T\in\T_h}\langle \bv_0-\bv_b,(\bbQ_h\rho)
\bn\rangle_\pT\nonumber\\
&=&-\sum_{T\in\T_h}(\bv_0,\nabla\rho)_T+\sum_{T\in\T_h}\langle \bv_0,\rho\bn\rangle_\pT-
\sum_{T\in\T_h}\langle \bv_0-\bv_b,(\bbQ_h\rho)\bn\rangle_\pT\nonumber\\
&=&-\sum_{T\in\T_h}(\bv_0,\nabla\rho)_T+\sum_{T\in\T_h}\langle \bv_0-\bv_b,\rho\bn\rangle_\pT
-\sum_{T\in\T_h}\langle \bv_0-\bv_b,(\bbQ_h\rho)\bn\rangle_\pT\nonumber\\
&=&-(\bv_0,\nabla\rho)+\sum_{T\in\T_h}\langle
\bv_0-\bv_b,(\rho-\bbQ_h\rho)\bn\rangle_\pT,
\end{eqnarray*}
which leads to
\begin{eqnarray}
(\bv_0,\nabla\rho)&=&-(\dw\bv,\bbQ_h\rho)+\sum_{T\in\T_h}\langle
\bv_0-\bv_b,(\rho-\bbQ_h\rho)\bn\rangle_\pT.\label{m2}
\end{eqnarray}

Next we test (\ref{new-moment}) by using $\bv_0$ in
$\bv=\{\bv_0,\bv_b\}\in V_h^0$ to obtain
\begin{equation}\label{m3}
-(\nabla\cdot(\nabla\bw),\bv_0)+(\nabla\rho,\bv_0)=(\eta,\bv_0).
\end{equation}
It follows from the usual integration by parts that
\[
-(\nabla\cdot(\nabla\bw),\bv_0)=\sum_{T\in\T_h}(\nabla\bw,\
\nabla\bv_0)_T-\sum_{T\in\T_h}\l\bv_0-\bv_b,\
\nabla\bw\cdot\bn\r_\pT,
\]
where we have used the fact that $\sum_{T\in\T_h}\langle\bv_b,
\nabla\bw\cdot\bn\rangle_\pT=0$. Using (\ref{m1}) and the equation
above, we have
\begin{eqnarray}
-(\nabla\cdot(\nabla\bw),\bv_0)&=&(\nabla_w (Q_h\bw),\nabla_w\bv)\nonumber\\
&-&\sum_{T\in\T_h}\l\bv_0-\bv_b,\nabla\bw\cdot\bn-{\bf
Q}_h(\nabla\bw)\cdot\bn\r_\pT.\label{m4}
\end{eqnarray}
Substituting (\ref{m2}) and (\ref{m4}) into (\ref{m3}) yields
\[
(\nabla_w(Q_h\bw),\nabla_w\bv)-(\dw\bv,\ \bbQ_h\rho)=(\eta,
\bv_0)+\ell_{\bw}(\bv)-\theta_\rho(\bv),
\]
which completes the proof of the lemma.
\end{proof}

The following is a result on the error equation for the weak
Galerkin finite element scheme (\ref{wg1})-(\ref{wg2}).

\begin{lemma}\label{Lemma:error-equation}
Let $\be_h$ and $\varepsilon_h$ be the error of the weak Galerkin
finite element solution arising from (\ref{wg1})-(\ref{wg2}), as
defined by (\ref{EQ:errors}). Then, we have
\begin{eqnarray}
a(\be_h,\ \bv)-b(\bv,\ \varepsilon_h)&=&\varphi_{\bu,p}(\bv),\label{ee1}\\
b(\be_h,\ q)&=&0,\label{ee2}
\end{eqnarray}
for all $\bv\in V_h^0$ and $q\in W_h$, where
$\varphi_{\bu,p}(\bv)=\ell_{\bu}(\bv)-\theta_p(\bv)+s(Q_h\bu,\bv)$
is a linear functional defined on $V_h^0$.
\end{lemma}

\begin{proof} Since $(\bu;p)$ satisfies the equation
(\ref{new-moment}) with $\eta=\bbf$, then from Lemma
\ref{Lemma:QhuEquation} we have
\[
(\nabla_w(Q_h\bu),\nabla_w\bv)-(\dw\bv,\bbQ_hp)=(\bbf,
\bv_0)+\ell_{\bu}(\bv)-\theta_p(\bv).
\]
Adding $s(Q_h\bu,\bv)$ to both side of the above equation gives
\begin{equation}\label{m5}
a(Q_h\bu,\bv)-b(\bv,\bbQ_hp)=(\bbf,\bv_0)+\ell_{\bu}(\bv)-\theta_p(\bv)
+s(Q_h\bu,\bv).
\end{equation}
The difference of (\ref{m5}) and (\ref{wg1}) yields the following equation,
\begin{eqnarray*}
a(\be_h,\bv)-b(\bv,\varepsilon_h)=\ell_{\bu}(\bv)-\theta_p(\bv)+s(Q_h\bu,\bv)
\end{eqnarray*}
for all $\bv\in V_h^0$, where
$\be_h=\{\be_0,\;\be_b\}=\{Q_0\bu-\bu_0,\;Q_b\bu-\bu_b\}$ and
$\varepsilon_h=\bbQ_hp-p_h$. This completes the derivation of
(\ref{ee1}).

As to (\ref{ee2}), we test equation (\ref{cont}) by $q\in W_h$ and
use (\ref{key11}) to obtain
\begin{equation}\label{m6}
0=(\nabla\cdot\bu,q)=(\nabla_w\cdot Q_h\bu,q).
\end{equation}
The difference of (\ref{m6}) and (\ref{wg2}) yields the following
equation
\begin{eqnarray*}
b(\be_h,q)=0
\end{eqnarray*}
for all $q\in W_h$. This completes the derivation of (\ref{ee2}).
\end{proof}

\section{Error Estimates}\label{Section:error-analysis}
In this section, we shall establish optimal order error estimates
for the velocity approximation $\bu_h$ in a norm that is equivalent
to the usual $H^1$-norm, and for the pressure approximation $p_h$ in
the standard $L^2$ norm. In addition, we shall derive an error
estimate for $\bu_h$ in the standard $L^2$ norm by applying the
usual duality argument in finite element error analysis.

\begin{theorem}\label{h1-bd}
Let $(\bu; p)\in  [H_0^1(\Omega)\cap H^{k+1}(\Omega)]^d\times
(L_0^2(\Omega)\cap H^{k}(\Omega))$ with $k\ge 1$ and $(\bu_h;p_h)\in
V_h\times W_h$ be the solution of (\ref{moment})-(\ref{bc}) and
(\ref{wg1})-(\ref{wg2}), respectively. Then, the following error
estimate holds true
\begin{eqnarray}
\3bar Q_h\bu-\bu_h\3bar+\|\bbQ_hp-p_h\|&\le& Ch^{k}(\|\bu\|_{k+1}+\|p\|_{k}).\label{err1}
\end{eqnarray}
\end{theorem}

\smallskip

\begin{proof}
By letting $\bv=\be_h$ in (\ref{ee1}) and $q=\varepsilon_h$ in
(\ref{ee2}) and adding the two resulting equations, we have
\begin{eqnarray}
\3bar \be_h\3bar^2&=&\varphi_{\bu, p}(\be_h).\label{main}
\end{eqnarray}
It then follows from (\ref{mmm1})-(\ref{mmm3}) (see Appendix A) that
\begin{equation}\label{b-u}
\3bar \be_h\3bar^2 \le Ch^{k}(\|\bu\|_{k+1}+\|p\|_{k})\3bar \be_h\3bar,
\end{equation}
which implies the first part of (\ref{err1}). To estimate
$\|\varepsilon_h\|$, we have from (\ref{ee1}) that
\[
b(\bv,\varepsilon_h)=a(\be_h,\bv)-\varphi_{\bu, p}(\bv).
\]
Using the equation above, (\ref{bd}), (\ref{b-u}) and
(\ref{mmm1})-(\ref{mmm3}), we arrive at
\[
|b(\bv,\varepsilon_h)|\le
Ch^{k}(\|\bu\|_{k+1}+\|p\|_{k})\3bar\bv\3bar.
\]
Combining the above estimate with the {\em inf-sup} condition
(\ref{inf-sup}) gives
\[
\|\varepsilon_h\|\le Ch^{k}(\|\bu\|_{k+1}+\|p\|_{k}),
\]
which yields the desired estimate (\ref{err1}).
\end{proof}

\medskip

In the rest of this section, we shall derive an $L^2$-error estimate
for the velocity approximation through a duality argument. To this
end, consider the problem of seeking $(\psi;\xi)$ such that
\begin{eqnarray}
-\Delta\psi+\nabla \xi&=\be_0 &\quad \mbox{in}\;\Omega,\label{dual-m}\\
\nabla\cdot\psi&=0 &\quad\mbox{in}\;\Omega,\label{dual-c}\\
\psi&= 0 &\quad\mbox{on}\;\partial\Omega.\label{dual-bc}
\end{eqnarray}
Assume that the dual problem has the $[H^{2}(\Omega)]^d\times
H^1(\Omega)$-regularity property in the sense that the solution
$(\psi; \xi)\in [H^{2}(\Omega)]^d\times H^1(\Omega)$ and the
following a priori estimate holds true:
\begin{equation}\label{reg}
\|\psi\|_{2}+\|\xi\|_1\le C\|\be_0\|.
\end{equation}

\medskip
\begin{theorem}
Let $(\bu; p)\in  [H_0^1(\Omega)\cap H^{k+1}(\Omega)]^d\times
(L^2_0(\Omega)\cap H^{k}(\Omega))$ with $k\ge 1$ and $(\bu_h;p_h)\in
V_h\times W_h$ be the solution of (\ref{moment})-(\ref{bc}) and
(\ref{wg1})-(\ref{wg2}), respectively.  Then, the following optimal
order error estimate holds true
\begin{equation}\label{l2-error}
\|Q_0\bu-\bu_0\|\le Ch^{k+1}(\|\bu\|_{k+1}+\|p\|_{k}).
\end{equation}
\end{theorem}

\begin{proof} Since $(\psi;\xi)$ satisfies the equation
(\ref{new-moment}) with $\eta=\be_0=Q_0\bu-\bu_0$, then from
(\ref{sf-01}) we have
\begin{eqnarray*}
(\nabla_wQ_h\psi,\nabla_w\bv)-(\nabla_w\cdot\bv,\bbQ_h\xi)=(\be_0,\bv_0)+\ell_\psi(\bv)
-\theta_\xi(\bv),\qquad \forall \bv\in V_h^0.
\end{eqnarray*}
In particular, by letting $\bv=\be_h$ we obtain
\begin{eqnarray*}
\|\be_0\|^2&=&(\nabla_wQ_h\psi,\nabla_w\be_h)-(\nabla_w\cdot\be_h,\bbQ_h\xi)-\ell_\psi(\be_h)
 +\theta_\xi(\be_h).
\end{eqnarray*}
Adding and subtracting $s(Q_h\psi,\be_h)$ in the equation above
yields
\begin{eqnarray*}
\|\be_0\|^2&=&a(Q_h
\psi,\be_h)-b(\be_h,\bbQ_h\xi)-\varphi_{\psi,\xi}(\be_h),
\end{eqnarray*}
where $\varphi_{\psi,\xi}(\bv)=\ell_\psi(\be_h)-\theta_\xi(\be_h)
+s(Q_h\psi,\be_h)$. It follows from (\ref{ee2}), (\ref{dual-c}) and
(\ref{m6}) that
$$
b(\be_h,\bbQ_h\xi)=0, \quad b(Q_h\psi,\varepsilon_h)=0.
$$
Combining the above two equations gives
\begin{eqnarray*}
\|\be_0\|^2&=&a(\be_h,Q_h\psi)-b(Q_h\psi,\varepsilon_h)-\varphi_{\psi,\xi}(\be_h).
\end{eqnarray*}
Using (\ref{ee1}) and the equation above, we have
\begin{equation}
\|\be_0\|^2=\varphi_{\bu,p}(Q_h\psi)-\varphi_{\psi,\xi}(\be_h).\label{d1}
\end{equation}

To estimate the two terms on the right hand side of (\ref{d1}), we
use the inequalities (\ref{mmm1})-(\ref{mmm3}) with
$(\bw;\rho)=(\psi;\xi)$, $\bv=\be_h$, and $r=1$ to obtain
\begin{equation}\label{sf-10}
|\varphi_{\psi,\xi}(\be_h)|\leq C h
(\|\psi\|_2+\|\xi\|_1)\3bar\be_h\3bar\leq Ch \3bar\be_h\3bar \
\|\be_0\|,
\end{equation}
where we have used the regularity assumption (\ref{reg}). Each of
the terms in $\varphi_{\bu,p}(Q_h\psi)$ can be handled as follows.

\begin{enumerate}
\item[(i)] For the stability term $s(Q_h\bu,Q_h\psi)$, we use the
definition of $Q_b$ and (\ref{trace}) to obtain
\begin{eqnarray*}
|s(Q_h\bu,\ Q_h\psi)|&=&\left|\sum_{T\in\T_h} h^{-1}_T\langle Q_b(Q_0\bu-\bu),\;
Q_b(Q_0\psi-\psi)\rangle_\pT\right|\\
&\le&\left(\sum_{T\in\T_h}h^{-1}_T\|Q_0\bu-\bu\|^2_{\pT}\right)^{1/2} \left(\sum_{T\in\T_h}h^{-1}_T\|Q_0\psi-\psi\|^2_{\pT}\right)^{1/2}\\
&\le& Ch^{k+1}\|\bu\|_{k+1}\|\psi\|_{2}.
\end{eqnarray*}

\item[(ii)] For the term $\ell_{\bu}(Q_h\psi)$, we first use the definition of $Q_b$ and
the fact that $\psi=0$ on $\partial\Omega$ to obtain
\begin{eqnarray*}
\sum_{T\in\T_h}\l \psi-Q_b\psi,\nabla\bu\cdot\bn-{\bf
Q}_h(\nabla\bu) \cdot\bn\r_\pT = \sum_{T\in\T_h}\l
\psi-Q_b\psi,\nabla\bu\cdot\bn\r_\pT = 0.
\end{eqnarray*}
Thus,
\begin{eqnarray*}
|\ell_{\bu}(Q_h\psi)|&=&\left|\sum_{T\in\T_h}\l
Q_0\psi-Q_b\psi,\nabla\bu\cdot\bn-{\bf Q}_h(\nabla\bu)
\cdot\bn\r_\pT\right|\\
&=&\left|\sum_{T\in\T_h}\l Q_0\psi-\psi,\nabla\bu\cdot\bn-{\bf
Q}_h(\nabla\bu)
\cdot\bn\r_\pT\right|\\
&\le& \left(\sum_{T\in\T_h}h_T\|\nabla\bu\cdot\bn-{\bf Q}_h(\nabla\bu)\cdot\bn\|_\pT^2\right)^{1/2}
\left(\sum_{T\in\T_h}h^{-1}_T\|Q_0\psi-\psi\|^2_{\pT}\right)^{1/2}\\
&\le& Ch^{k+1}\|\bu\|_{k+1}\|\psi\|_{2}.
\end{eqnarray*}

\item[(iii)] For the term $\theta_{p}(Q_h\psi)$, we first use the definition of $Q_b$ and
the fact that $\psi=0$ on $\partial\Omega$ to obtain
\begin{eqnarray*}
\sum_{T\in\T_h}\langle
\psi-Q_b\psi,(p-\bbQ_hp)\bn\rangle_\pT=\sum_{T\in\T_h}\langle
\psi-Q_b\psi,p\bn\rangle_\pT=0.
\end{eqnarray*}
Thus, from (\ref{trace}) and (\ref{Lh}) we obtain
\begin{eqnarray*}
|\theta_p(Q_h\psi)|&=&\left|\sum_{T\in\T_h}\langle Q_0\psi-Q_b\psi,(p-\bbQ_hp)\bn\rangle_\pT\right|\\
&=&\left|\sum_{T\in\T_h}\langle
Q_0\psi-\psi,(p-\bbQ_hp)\bn\rangle_\pT\right|\\
&\le& \left(\sum_{T\in\T_h}h_T\| p-\bbQ_hp\|_\pT^2\right)^{1/2}\left(\sum_{T\in\T_h}h^{-1}_T
\|Q_0\psi-\psi\|^2_{\pT}\right)^{1/2}\\
&\le& Ch^{k+1}\|p\|_{k} \|\psi\|_{2}.
\end{eqnarray*}

\end{enumerate}

The three estimates in (i), (ii), (iii), and the regularity
(\ref{reg}) collectively yield
\begin{eqnarray}
|\varphi_{\bu,p}(Q_h\psi)|&\leq & C h^{k+1}
(\|\bu\|_{k+1}+\|p\|_k)\|\psi\|_2\nonumber\\
&\leq & C h^{k+1} (\|\bu\|_{k+1}+\|p\|_k)\|\be_0\|.\label{sf-15}
\end{eqnarray}

Finally, substituting (\ref{sf-10}) and (\ref{sf-15}) into
(\ref{d1}) gives
$$
\|\be_0\|^2\le Ch^{k+1}(\|\bu\|_{k+1}+\|p\|_{k})\|\be_0\| + Ch
\3bar\be_h\3bar \ \|\be_0\|.
$$
It follows that
$$
\|\be_0\|\le Ch^{k+1}(\|\bu\|_{k+1}+\|p\|_{k}) + Ch \3bar\be_h\3bar,
$$
which, together with Theorem \ref{h1-bd}, completes the proof of the
theorem.
\end{proof}

\section{Concluding Remarks}\label{Section:concluding}

This paper introduced a new finite element method for the Stokes
equations by using the general concept of weak Galerkin. The scheme
is applicable to finite element partitions of arbitrary polygon or
polyhedra. The paper has laid a solid theoretical foundation for the
stability and convergence of the weak Galerkin method. There are,
however, many open issues that need to be investigated in future
work. Here we would like to list a few for interested readers to
consider: (1) how the discretized linear systems can be solved
efficiently by using techniques such as domain decomposition and
multigrids? (2) can the weak Galerkin scheme for the Stokes
equations be hybridized? If so, how such a hybridization may help in
variable reduction and solution solving? and (3) what
superconvergence can one develop for the weak Galerkin method? (4)
is the weak Galerkin method more competitive than other existing
finite element schemes in practical computation? (5) what stability
do weak Galerkin methods have in other norms such as $L^p, p>1$?

\appendix
\section*{Local $L^2$ Projections}\label{appendix-1}

\setcounter{section}{1}

In this Appendix, we shall provide some technical results regarding
approximation properties for the $L^2$ projection operators $Q_h$,
${\bf Q}_h$, and $\bbQ_h$. These estimates have been employed in
previous sections to yield various error estimates for the weak
Galerkin finite element solution of the Stokes problem arising from
the scheme (\ref{wg1})-(\ref{wg2}).

\begin{lemma}\label{lem-1}
Let $\T_h$ be a finite element partition of $\Omega$ satisfying the
shape regularity assumption as specified in \cite{wy-mixed} and
$\bw\in [H^{r+1}(\Omega)]^d$ and $\rho\in H^{r}(\Omega)$ with $1\le
r\le k$. Then, for $0\le s\le 1$ we have
\begin{eqnarray}
&&\sum_{T\in\T_h} h^{2s}_T\|\bw-Q_0\bw\|_{T,s}^2\le h^{2(r+1)}
\|\bw\|^2_{r+1},\label{Qh}\\
&&\sum_{T\in\T_h} h^{2s}_T\|\nabla\bw-{\bf Q}_h(\nabla\bw)\|^2_{T,s}
\le Ch^{2r}
\|\bw\|^2_{r+1},\label{Rh}\\
&&\sum_{T\in\T_h} h^{2s}_T\|\rho-\bbQ_h\rho\|^2_{T,s} \le
Ch^{2r}\|\rho\|^2_{r}.\label{Lh}
\end{eqnarray}
Here $C$ denotes a generic constant independent of the meshsize $h$
and the functions in the estimates.
\end{lemma}

A proof of the lemma can be found in \cite{wy-mixed}, which is based
on some technical inequalities for functions defined on
polygon/polyhedral elements with shape regularity. We emphasize that
the approximation error estimates in Lemma \ref{lem-1} hold true
when the underlying mesh $\T_h$ consists of arbitrary polygons or
polyhedra with shape regularity as detailed in \cite{wy-mixed} and
\cite{mwy}.

\smallskip
Let $T$ be an element with $e$ as an edge/face. For any function
$g\in H^1(T)$, the following trace inequality has been proved to be
valid for general meshes satisfying the shape regular assumptions
detailed in \cite{wy-mixed}:
\begin{equation}\label{trace}
\|g\|_{e}^2 \leq C \left( h_T^{-1} \|g\|_T^2 + h_T \|\nabla
g\|_{T}^2\right).
\end{equation}

\begin{lemma}\label{lem-11}
For any $\bv=\{\bv_0,\bv_b\}\in V_h$, we have
\begin{equation}\label{mad}
\sum_{T\in\T_h}\|\nabla \bv_0\|_T^2\le C\3bar v\3bar^2.
\end{equation}
\end{lemma}

\begin{proof}
For any $\bv=\{\bv_0,\bv_b\}\in V_h$, it follows from the
integration by parts and the definitions of weak gradient and $Q_b$,
\begin{eqnarray*}
(\nabla \bv_0,\nabla \bv_0)_T&=& -(\bv_0,\nabla\cdot\nabla \bv_0)_T+\l \bv_0, \nabla \bv_0\cdot\bn\r_\pT\\
&=& -(\bv_0,\nabla\cdot\nabla \bv_0)_T+\l \bv_b, \nabla \bv_0\cdot\bn\r_\pT+\l \bv_0-\bv_b, \nabla \bv_0\cdot\bn\r_\pT\\
&=&(\nabla_w\bv,\nabla \bv_0)_T+\l Q_b\bv_0-\bv_b, \nabla
\bv_0\cdot\bn\r_\pT.
\end{eqnarray*}
By applying the trace inequality (\ref{trace}) and the inverse
inequality to the equation above, we obtain
\[
\|\nabla \bv_0\|_T^2\le C( \|\nabla_w\bv\|_T\|\nabla
\bv_0\|_T+h_T^{-\frac12}\|Q_b\bv_0-\bv_b\|_\pT\|\nabla \bv_0\|_T).
\]
Thus,
$$
\|\nabla \bv_0\|_T^2\le
C(\|\nabla_w\bv\|_T^2+h_T^{-1}\|Q_b\bv_0-\bv_b\|_\pT^2),
$$
which gives rise to (\ref{mad}) after a summation over all
$T\in\T_h$.
\end{proof}

\smallskip

\begin{lemma}
Let $1\le r\le k$ and $\bw\in [H^{r+1}(\Omega)]^d$ and $\rho\in
H^r(\Omega)$ and $\bv\in V_h$. Assume that the finite element
partition $\T_h$ is shape regular. Then, the following estimates
hold true
\begin{eqnarray}
|s(Q_h\bw,\bv)|&\le& Ch^{r}\|\bw\|_{r+1}\3bar \bv\3bar,\label{mmm1}\\
|\ell_{\bw}(\bv)|&\le& Ch^{r}\|\bw\|_{r+1}\3bar \bv\3bar,\label{mmm2}\\
|\theta_\rho(\bv)|&\le& Ch^{r}\|\rho\|_{r}\3bar
\bv\3bar,\label{mmm3}
\end{eqnarray}
where $\ell_{\bw}(\cdot)$ and $\ell_\rho(\cdot)$ are two linear
functionals on $V_h$ given by
\begin{eqnarray}
\ell_{\bw}(\bv)&=&\sum_{T\in\T_h}\l\bv_0-\bv_b,\nabla\bw\cdot\bn-{\bf
Q}_h(\nabla\bw)\cdot\bn\r_\pT,\label{l1-inLemma}\\
\theta_\rho(\bv)&=&\sum_{T\in\T_h}\langle
\bv_0-\bv_b,(\rho-\bbQ_h\rho)\bn\rangle_\pT.\label{l2-inLemma}
\end{eqnarray}
\end{lemma}

\begin{proof}
Using the definition of $Q_b$, (\ref{trace}) and (\ref{Qh}), we have
\begin{eqnarray*}
|s(Q_h\bw,\ \bv)|&=&\left|\sum_{T\in\T_h} h_T^{-1}\langle Q_b(Q_0\bw)-Q_b\bw,\; Q_b\bv_0-\bv_b\rangle_\pT\right|\\
&=&\left|\sum_{T\in\T_h} h_T^{-1}\langle Q_b(Q_0\bw-\bw),\; Q_b\bv_0-\bv_b\rangle_\pT\right|\\
&=&\left|\sum_{T\in\T_h} h_T^{-1} \langle Q_0\bw-\bw,\; Q_b\bv_0-\bv_b\rangle_\pT\right|\\
&\le& \left(\sum_{T\in\T_h}(h_T^{-2}\|Q_0\bw-\bw\|_T^2+\|\nabla (Q_0\bw-\bw)\|_T^2)\right)^{1/2}\\
& & \ \left(\sum_{T\in\T_h}h_T^{-1}\|Q_b\bv_0-\bv_b\|^2_{\pT}\right)^{1/2}\\
&\le& Ch^{r}\|\bw\|_{r+1}\3bar \bv\3bar.
\end{eqnarray*}
It follows from (\ref{trace}) and (\ref{Rh}) that
\begin{eqnarray*}
|\ell_{\bw}(\bv)|&=&\left|\sum_{T\in\T_h}\l\bv_0-\bv_b,\ \nabla\bw\cdot\bn-{\bf Q}_h(\nabla\bw)\cdot\bn\r_\pT\right|\\
&\le&\left|\sum_{T\in\T_h}\l\bv_0-Q_b\bv_0,\ \nabla\bw\cdot\bn-{\bf Q}_h(\nabla\bw)\cdot\bn\r_\pT\right|\\
&+&\left|\sum_{T\in\T_h}\l Q_b\bv_0-\bv_b,\ \nabla\bw\cdot\bn-{\bf
Q}_h(\nabla\bw)\cdot\bn\r_\pT\right|.
\end{eqnarray*}
To estimate the first term on the righ-hand side of the above
inequality, we use (\ref{trace}), (\ref{Rh}), (\ref{mad}) and the
inverse inequality to obtain
\begin{eqnarray*}
&&\left|\sum_{T\in\T_h}\l\bv_0-Q_b\bv_0,\ \nabla\bw\cdot\bn-{\bf Q}_h(\nabla\bw)\cdot\bn\r_\pT\right|\\
\le && C\sum_{T\in {\cal T}_h} h_T\|\nabla\bw\cdot\bn-{\bf Q}_h(\nabla\bw)\cdot\bn\|_\pT\|\nabla\bv_0\|_\pT\\
\le && C\left(\sum_{T\in {\cal T}_h}h_T\|\nabla\bw\cdot\bn-{\bf Q}_h(\nabla\bw)\cdot\bn\|_\pT^2\right)^{1/2}\left(\sum_{T\in {\cal T}_h}\|\nabla\bv_0\|_\pT^2\right)^{1/2}\\
\le && Ch^{r}\|\bw\|_{r+1}\3bar \bv\3bar.
\end{eqnarray*}
Similarly, for the second term, we have
\begin{eqnarray*}
&&\left|\sum_{T\in\T_h}\l Q_b\bv_0-\bv_b,\ \nabla\bw\cdot\bn-{\bf Q}_h(\nabla\bw)\cdot\bn\r_\pT\right|\\
\le && C\left(\sum_{T\in {\cal T}_h}h_T\|\nabla\bw\cdot\bn-{\bf Q}_h(\nabla\bw)\cdot\bn\|_\pT^2\right)^{1/2}\left(\sum_{T\in {\cal T}_h}h_T^{-1}\|Q_b\bv_0-\bv_b\|_\pT^2\right)^{1/2}\\
\le && Ch^{r}\|\bw\|_{r+1}\3bar \bv\3bar.
\end{eqnarray*}
The estimate (\ref{mmm2}) is verified by combining the above three
estimates.

The same technique for proving (\ref{mmm2}) can be applied to yield
the following estimate.
\begin{eqnarray*}
|\theta_\rho(\bv)|&=&\left|\sum_{T\in\T_h}\langle \bv_0-\bv_b,\ (\rho-\bbQ_h\rho)\bn\rangle_\pT\right|\\
&\le& Ch^{r}\|\rho\|_{r}\3bar \bv\3bar.
\end{eqnarray*}
This completes the proof of the lemma.
\end{proof}


\newpage

\begin{thebibliography}{99}

\bibitem{babuska}
{\sc I. Babu\u{s}ka},
{\em The finite element method with Lagrangian multiplier},
Numer. Math., 20 (1973), pp.~179--192.

\bibitem{bs}
{\sc S. Brenner and R. Scott},
{\em Mathematical theory of finite element methods}, Springer, 2002.

\bibitem{brezzi}
{\sc F. Brezzi},
{\em On the existence, uniqueness, and approximation of saddle point
problems arising from Lagrangian multipliers},
RAIRO, Anal. Num\'{e}r.,  2 (1974), pp.~129--151.


\bibitem{bf}
{\sc F. Brezzi and M. Fortin},
\textit{Mixed and Hybrid Finite Elements},
Springer-Verlag, New York, 1991.


\bibitem{cr}
{\sc M. Crouzeix and P. A. Raviart},  \textit{Conforming and nonconforming
finite element methods for solving the stationary Stokes
equations}, RAIRO Anal. Numer., 7 (1973), pp.~33--76.

\bibitem{gr}
{\sc V. Girault and P.A. Raviart},
\textit{Finite Element Methods for
the Navier-Stokes Equations: Theory and Algorithms}, Springer-Verlag, Berlin, 1986.

\bibitem{gun}
{\sc M. D. Gunzburger},
\textit{Finite Element Methods for Viscous Incompressible Flows,
A Guide to Theory, Practice and Algorithms}, Academic,
San Diego, 1989.

\bibitem{mwy}
{\sc L. Mu, J. Wang, and X. Ye}, {\em Weak Galerkin finite element
methods on Polytopal Meshes}, arXiv:1204.3655v2.

\bibitem{wy}
{\sc J. Wang and X. Ye}, {\em A weak Galerkin finite element method
for second-order elliptic problems},  J. Comp. and Appl. Math, 241 (2013) 103-115.

\bibitem{wy-mixed}
{\sc J. Wang and X. Ye}, {\em A Weak Galerkin mixed finite element method for
second-order elliptic problems}, arXiv:1202.3655v1.

\end{thebibliography}
\end{document}